\documentclass[11pt,letterpaper]{amsart}
\usepackage[utf8]{inputenc}
\usepackage{graphicx}
\graphicspath{ {./images/} }
\usepackage{tikz}
\usepackage{tikz-cd}
\usepackage{enumitem}
\usepackage{bm}
\usepackage{mathtools} 
\usepackage{extarrows} 
\usepackage{amsmath}
\usepackage{amssymb}
\usepackage[utf8]{inputenc}
\usepackage{amsthm}
\usepackage{tabularx}
\usepackage{mathrsfs}
\usepackage[
sorting=nyt]{biblatex}

\newtheorem{theorem}{Theorem}[section]
\newtheorem{lemma}[theorem]{Lemma}
\newtheorem{corollary}[theorem]{Corollary}
\newtheorem{defn}{Definition}[section]
\newtheorem{prop}{Proposition}
\newtheorem{remark}{Remark}

\newcommand{\bs}{\backslash}

\numberwithin{equation}{section}

\title{Cohomology at Infinity and the Well-Tempered Complex}
\author{Dylan Galt \and Mark McConnell}

\begin{document}
\maketitle
\begin{abstract}
   We prove the existence of a sequence of commutative diagrams generalizing the results on cohomology at infinity described in \cite{avner_cohomology_1997} to the context of the well-tempered complex introduced in \cite{mcconnell_computing_2020}. Our main theorem provides a method for computing in finite terms the action of Hecke operators on the cohomology of the Borel-Serre boundary for the $\textup{SL}_n$ symmetric space.
\end{abstract}
\section*{Introduction}
Let $\mathbf{G}$ be a reductive algebraic group defined over
$\mathbb{Q}$, and let $X$ be the symmetric space for
$\mathbf{G}(\mathbb{R})$.  Let $\Gamma\subseteq\mathbf{G}(\mathbb{Q})$
be an arithmetic subgroup. In general, the cohomology of the quotient
$\Gamma\backslash X$ is difficult to compute.  In the case
$\mathbf{G}=\textup{SL}_n$, there has been considerable work involving the well-rounded retract $\tilde{W}$ due to Ash
\cite{ash_small-dimensional_1984} to compute the cohomology of
$\Gamma\backslash X$, and also to compute the Hecke action on the
cohomology in certain degrees. (A short list of references relevant to
this paper is \cite{Manin, Stein, AR, Gun, avner_ash_cohomology_2010,
  ash_cohomology_2020, AGMY, GunMY}.  See
\cite{harder_cohomology_2017} for an approach in principle for
general~$\mathbf{G}$.)  The well-rounded retract is a cell complex
$\tilde{W}\subset X$ sitting inside the symmetric space as a spine on which
$\Gamma$ acts cellularly, and there is an explicit,
$\Gamma$-equivariant deformation retraction $X\rightarrow \tilde{W}$ called
the well-rounded retraction.

Given this well-rounded retract, there are several natural questions. A first is whether the map $X\rightarrow \tilde{W}$ extends as a $\Gamma$-equivariant deformation retraction $\overline{X}\rightarrow \tilde{W}$ to the Borel-Serre bordification $\overline{X}$ of $X$. This question was answered in the affirmative in \cite{avner_cohomology_1997}.  This shows that the well-rounded retract can be used to compute the canonical restriction map $H^\ast_\Gamma(\overline{X})\rightarrow H^\ast_\Gamma(\partial\overline{X})$ in equivariant cohomology for any coefficient module.

A second question is how to compute the Hecke action on the equivariant cohomology
$H^i_\Gamma(X)$ in all degrees~$i$. The
paper~\cite{mcconnell_computing_2020} extended the well-rounded
retract to the \emph{well-tempered complex}, a continuous family of cell
complexes parametrized by a real parameter~$\tau$ called the
temperament.  At each~$\tau$, the slice of the well-tempered
complex is a certain well-rounded retract as
in~\cite{ash_small-dimensional_1984}.  $\Gamma$ acts on the
well-tempered complex.  There are finitely many critical temperaments
where the cell-complex structure changes abruptly.  A Hecke operator
acting on $H^\ast_\Gamma(X)$ can be computed in finite terms as a
series of isomorphisms and quasi-isomorphisms between the slices of
the well-tempered complex at these critical temperaments.

It is natural to ask whether the techniques of \cite{avner_cohomology_1997} and \cite{mcconnell_computing_2020} can be combined to show that the well-tempered complex extends to the Borel-Serre boundary, yielding a generalization of the algorithm developed in \cite{mcconnell_computing_2020} to allow for the computation of Hecke operators on the cohomology of the Borel-Serre boundary.  In the present paper, we answer this question in the affirmative.

A key insight in \cite{avner_cohomology_1997} was that in order to extend the well-rounded retraction to the Borel-Serre boundary, one undertakes two steps. The first is to construct a family of neighborhoods of the Borel-Serre boundary components, in correspondence with $\mathbb{Q}$-parabolic subgroups $P$, each of which fibres over the corner associated to the given choice of $P$. One then shows that the well-rounded retraction is constant on these fibres. Intuitively, this gives a way to pull the Borel-Serre boundary into the interior of the symmetric space. The second step, using ideas of Saper \cite{saper_tilings_1997}, is to choose a $\Gamma$-equivariant tiling of the symmetric space $X$ with a $\Gamma$-invariant central tile that intersects each of these boundary neighborhoods. Having pulled the boundary into the interior, one may then apply the well-rounded retraction restricted to this central tile.

The main geometric construction of this paper is a one-parameter
family of boundary neighborhoods that vary continuously with the
temperament $\tau$ of the well-tempered complex. This allows us to
mimic step one of \cite{avner_cohomology_1997}, taking the changing
temperament into account. We then use a result from
\cite{saper_tilings_1997} to choose a one-parameter family of central
tiles that allows us to carry out step two. Our main theorem is the
existence of a sequence of cubical commutative diagrams which combine
the spectral sequences studied in \cite{avner_cohomology_1997} with
the algorithm of \cite{mcconnell_computing_2020}.  The
Hecke operators are found using a series of isomorphisms and quasi-isomorphisms
between the spectral sequences for the slices of the well-tempered
complex at the critical temperaments.

The result provides a method for computing in finite terms the action of Hecke operators on the cohomology of the Borel-Serre boundary of the symmetric space for $\textup{SL}_n$. We hope this will allow for a better understanding of interior cohomology and ghost classes, which are relevant to Langlands-type conjectures on the $L$-functions of certain Eisenstein series and cuspidal automorphic forms. In particular, one would hope to better understand whether certain cusp forms are lifts of those from algebraic subgroups of the special linear group, such as $\textup{Sp}_4$. This relates to our work on the cohomology of $\textup{Sp}_4$ in \cite{galt_computing_2021}, which constructs an acyclic cell complex for $\textup{Sp}_4$ and, using results from \cite{mcconnell_computing_2020}, yields the beginnings of an algorithm for the computation of Hecke operators on the cohomology of Siegel modular threefolds.

An outline of this paper is as follows. In Sections \ref{The
  Well-Rounded Retract}--\ref{Boundary Neighborhoods for a Fixed Set
  of Weights}, we present background material from
\cite{avner_cohomology_1997}. We present this at the level of
generality considered in that paper, although the case $k=\mathbb{Q}$, rather than that of arbitrary number fields, is often ultimately of most interest. With this in mind, the reader may take $k=\mathbb{Q}$ throughout the paper, for simplicity, and no essential understanding will be sacrificed. In Section~\ref{The Well-Tempered Complex},
we recall the construction of the well-tempered complex from
\cite{mcconnell_computing_2020}. In Section~\ref{The Borel-Serre
  Boundary}, we give our main geometric construction of boundary
neighborhoods that vary continuously under the change of
temperament. Section~\ref{A Family of Central Tiles} recalls results from
\cite{saper_tilings_1997} and proves the existence of a one-parameter
family of central tiles that meets our boundary neighborhoods at each
temperament. In Section~\ref{Extending to the Boundary}, we summarize
the implications of these geometric constructions, proving the
existence of a one-parameter family of deformation
retractions. Section~\ref{Cohomology at Infinity} outlines the
spectral sequences that generalize those of
\cite{avner_cohomology_1997}, and in Section~\ref{Hecke
  Correspondences} we prove our main theorem.  Section~\ref{Comments
  on Implementation} contains practical comments on implementation and
computation.

\section{The Well-Rounded Retract} \label{The Well-Rounded Retract}
Let $k$ be a number field of degree $d$ over $\mathbb{Q}$ with ring of integers $\mathcal{O}$. Let $S=k\otimes_\mathbb{Q}\mathbb{R}$ and let $G=\textup{GL}_n(S)$. Fix an integer $n\geq 1$. There is a natural right\footnote{Consistent with \cite{mcconnell_computing_2020}, we consider the space of row vectors $S^n$ as a \textit{right} $G$-module. Our arithmetic groups act on the left. This is consistent with many modern references on equivariant cohomology, for instance \cite{brown_cohomology_1982}, but reverses the convention of \cite{avner_cohomology_1997}.} action of $G$ on the space of row vectors in $S^n$. A \emph{homothety} is multiplication of $S^n$ by a positive real scalar; homotheties form a subgroup~$H$ of~$G$.  A \emph{$\mathbb{Z}$-lattice} in $k^n$ is a finitely-generated subgroup of $k^n$ containing a $\mathbb{Q}$-basis of $k^n$ that is a $\mathbb{Z}$-basis of the subgroup. Fix a $\mathbb{Z}$-lattice $L_0$. Denote by $\mathcal{O}_0\subseteq k$ the stabilizer of $L_0$ under the action of $G$ on $k^n\subset S^n$ and let $\Gamma_0$ be the arithmetic group whose elements are those $\gamma\in G$ satisfying $\gamma\cdot L_0=L_0$. Let $\Gamma$ be any arithmetic subgroup of finite index in $\Gamma_0$ (not necessarily torsion free).

\begin{defn}[\cite{avner_cohomology_1997}, Def.~2.3]
A marked lattice inside $S^n$ is a homothety class of functions $f:L_0\rightarrow S^n$ of the form $f(x)=xg$ for some $g\in G$, where $f$ is equivalent to $f\cdot h$ for any homothety $h\in H$.
\end{defn}

The space of marked lattices is identified with $Y'=G/H$. Each ordinary lattice $L$ arises as the image of our fixed lattice $L_0$ under some $f\in Y'$, so that the space $Y=\Gamma_0\backslash G/H$ is identified with the set of all integer lattices in $S^n$ fixed by the action of $\mathcal{O}_0$ and isomorphic as $\mathcal{O}_0$-modules to $L_0$. 

Let $K$ be a maximal compact subgroup of $G$, specifically the
subgroup of $G$ that preserves the inner product $\left<\cdot, \cdot
\right>$ on $S^n$ that is defined by summing over the archimedean
places of $k$ \cite[553]{avner_cohomology_1997}. Let $X=G/HK$
\cite[(2.3.1)]{avner_cohomology_1997}.  The quotient $Y/K$ is
$\Gamma_0\backslash X$.

A \textit{set of weights} for an arithmetic subgroup $\Gamma$ is a set function $\Phi: \Gamma\backslash\mathbb{P}(k^n)\rightarrow\mathbb{R}_{>0}$. Analogously, a \textit{set of weights} for a $\mathbb{Z}$-lattice $L$ is a $\Gamma$-invariant function $\Phi^L: L-\{0\}\rightarrow\mathbb{R}_{>0}$. Given $\Phi$, we can define $\Phi^{L_0}(x)=\Phi(k\cdot x)$ and $\Phi^L(xg)=\Phi(x)$, when $L=L_0 g$. If $f$ is a marked lattice, with $f(L_0)=L$, we have $\Phi^L(f(x))=\Phi(x)$.  To define the well-tempered complex, we will use a \textit{one-parameter family of weights} for $L_0$ with parameter $\tau$, which is a family, real-analytic in $\tau$, of $\Gamma$-invariant weights denoted
\begin{equation*}
    \Phi_\tau:(L_0-\{0\})\times I\rightarrow\mathbb{R}_{>0}.
\end{equation*}

Given a set of weights, we may define the notion of a well-rounded marked lattice. Indeed, with respect to weights $\Phi^L$, the \textit{arithmetic minimum} of a marked lattice $f$ for which $f(L_0)=L$ is the number
\begin{equation*}
    m(f)=\textup{min}\{\Phi^L(x)\left< x,x\right>\; | \; x\in L-\{0\}\}.
\end{equation*}
The \textit{minimal vectors} of $f$ are the set of $x\in L$ with weighted length $m(f)$.

\begin{defn}[\cite{avner_cohomology_1997}, Def.~2.5]
A marked lattice $f\in Y'$ is \emph{well rounded} if its set of minimal vectors forms an $S$-module basis for $S^n$.
\end{defn}

The set $\tilde{W}$ of well-rounded\footnote{As in \cite{mcconnell_computing_2020}, cell complexes are decorated with a tilde if and only if they live ``upstairs'' in the symmetric space.} elements in $X$ \cite[(2.5)]{avner_cohomology_1997} is the \textit{well-rounded retract} in $X$ with respect to the specified set of weights, and the quotient $\Gamma\backslash \tilde{W}$ is the well-rounded retract in $\Gamma\backslash X$. The well-rounded retraction is a map $r_t:Y'\times [0,1]\rightarrow Y'$ defined on a given marked lattice $f$ as follows. Fix a choice of weights $\Phi$. Suppose that $L=f(L_0)$ and that $f$ is chosen in its equivalence class modulo $H$ so that $m(f)=1$. For each $i=1,\cdots, n$, let $Y'_i=\{f\in Y'\: |\; \textup{rk}_S(S\cdot M(f))\geq i\}$. This forms a descending sequence ending in $Y'_n$, the set of well-rounded marked lattices. For each $i$, there is a deformation retraction $r^{(i)}_t:Y'_i\times [0,1]\rightarrow Y'_i$ with image $Y'_{i+1}$ that is equivariant with respect to the actions of $K$ and $\Gamma_0$. 

Let $f_1=f$ and $f_i=r^{(i-1)}_1(f_{i-1})$, inductively, and define a $\mathbb{Q}$-flag $\mathscr{M}=\{0\subseteq M^{(1)}\subseteq\cdots\subseteq M^{(n)}=k^n\}$ where $M^{(i)}$ is the $k$-span of $f^{-1}_i(M(f_i))$, recalling that $M(f_i)$ is the set of minimal vectors for $f_i$ with respect to our fixed set of weights. Then, by definition
\begin{equation*}
    r^{(i)}_t(f_i)=\left\{\begin{matrix}
f_i & \textup{if}\; f_i\in Y'_{i+1} \\
\varphi_{1+(\mu_i-1)t}\cdot f_i & \textup{if}\; f_i\in Y'_i-Y'_{i+1} \\
\end{matrix}\right.,
\end{equation*}
where $\mu_i$ is the unique smallest number in $(0,1)$ so that $m(\varphi_{\mu_i}\cdot f_i)=1$ and $\varphi_{\mu_i}$ is the $S$-linear map which is multiplication by $1$ in $f_i(M^{(i)})\otimes_\mathbb{Q}\mathbb{R}$ and by $\mu_i$ in its orthogonal complement. The well-rounded retraction $r_t$ applied to a given marked lattice $f$ is the composition
\begin{equation*}
    r_t(f)=r^{(i)}_{t(n-1)-(i-1)}\circ r^{(i-1)}_1\circ\cdots\circ r^{(1)}_1(f),
\end{equation*}
where $i$ is such that $t\in [(i-1)/(n-1), i/(n-1)]$. The flag $\mathscr{M}$ is called the flag of successive minima for the marked lattice $f$. Note that this depends on the choice of weights.

In \cite{avner_cohomology_1997}, only the trivial set of weights is considered. However, the essential construction of neighborhoods of the Borel-Serre boundary components that we now recall goes through unchanged if we fix a priori a (possibly nontrivial) set of weights. Later, we will examine the case of a one-parameter family of weights for $L_0$.

\section{The Borel-Serre Compactification} \label{The Borel-Serre Compactification}

In general, for $\mathbf{G}$ a reductive algebraic group defined over $\mathbb{Q}$, $G=\mathbf{G}(\mathbb{R})$ its group of real points, $X$ its symmetric space, and $\Gamma\subset\mathbf{G}(\mathbb{Q})$ an arithmetic subgroup, one can form the Borel-Serre compactification of the quotient $\Gamma\backslash X$. This smooth manifold-with-corners is obtained by adjoining to X pieces indexed by parabolics, then taking the quotient of this ``bigger'' space by $\Gamma$. We now specialize to $\mathbf{G}=\mathbf{GL}_n$ and review this construction.

The set $_\mathbb{Q}\mathbf{T}=\{\textup{diag}(a_1,\cdots, a_n)\}$ is a maximal $\mathbb{Q}$-split torus in $G$ with $\mathbb{Q}$-roots $\{a_ia^{-1}_j\; | \; 1\leq i\neq j\leq n\}$. Let $\Delta=\{a_ia^{-1}_{i+1}\;|\;i=1,\cdots, n-1\}$ be the fundamental system of simple roots for $_\mathbb{Q}\mathbf{T}$. For any subset $J\subseteq\Delta$, denote by $_\mathbb{Q}\mathbf{T}_J$ the subgroup of $_\mathbb{Q}\mathbf{T}$ consisting of elements $\{\textup{diag}(t_1,\cdots, t_n)\}$ for which the $i$th and $(i+1)$st diagonal entries are equal whenever $t_it^{-1}_{i+1}\in J$. Denote by $T_J$ the identity component of $_\mathbb{Q}\mathbf{T}_J(\mathbb{R})$. Let $\{e_1,\cdots, e_n\}$ be the standard $k$-basis for $k^n$. Construct a graph with vertices $\{1,\cdots, n\}$ having edges that join vertices $i$ and $i+1$ if and only if $a_ia^{-1}_{i+1}\in J$. Consider the flag $\mathscr{F}_J=\{0\subsetneq V_1\subsetneq \cdots \subsetneq V_\ell=k^n\}$ in which each $V_j$ is spanned by $\{e_1,\cdots, e_{v_1+\cdots +v_j}\}$, for $v_j$ the number of vertices of the $j$th component of the graph (ordered left to right). Let $\mathbf{P}_J$ be the subgroup of $\mathbf{G}$ for which $\mathbf{P}_J(\mathbb{Q})$ is the stabilizer of $\mathscr{F}_J$ in $\mathbf{G}(\mathbb{Q})$. Let $P_J=\mathbf{P}_J(\mathbb{R})$.

Given any $g\in \mathbf{G}(\mathbb{Q})$, the conjugate $P=gP_Jg^{-1}$ admits a Levi decomposition as a semidirect product $P=M_PT_PU_P$, where $T_P=gT_Jg^{-1}$, $M_PT_P$ is the Levi subgroup of $P$ stable under the Cartan involution of $G$, and $U_P\subset P$ is the unipotent radical of $P$. We can form the quotients $A_P=T_P/H$ and $A_J=T_J/H$. There is an explicit isomorphism $A_J\rightarrow (0,\infty)^{\ell-1}$, for $\ell=\#(\Delta-J)+1$, prescribed by the roots in $\Delta-J$, which determines a partial compactification $\overline{A}_J$ of $A_J$ corresponding to $(0,\infty]^{\ell-1}$ under the aforementioned isomorphism.\footnote{Here we use the convention in \cite{saper_tilings_1997} of adding $\infty$ rather than $0$ (as is done in \cite{avner_cohomology_1997}). This is a matter of taste, but will prove convenient when we apply Saper's work in Section~\ref{A Family of Central Tiles}.} 

Recall that there is a geodesic action of $A_J$ on $X$, which extends, via conjugation by $g\in\mathbf{G}(\mathbb{Q})$, to a geodesic action of $A_P$ on $X$. There is a corresponding partial compactification $\overline{A}_P$ of $A_P$ and one defines the corner associated to $P$ to be the fibre product $X\times^{A_P}\overline{A}_P$. There is an inclusion-reversing correspondence between parabolic subgroups and associated corners. Inside each corner is a boundary face $e(P)=X\times ^{A_P}\{\infty\}^{\ell-1}$ and the Borel-Serre bordification $\overline{X}$ of $X$ is defined as a set to be the disjoint union of $X$ and the $e(P)$, as $P$ ranges over all the proper $\mathbb{Q}$-parabolic subgroups of $G$. One can topologize $\overline{X}$ as a Hausdorff space admitting a proper $\Gamma$-action with compact quotient $\Gamma\backslash\overline{X}$, the Borel-Serre compactification of $\Gamma\backslash X$.
%%Put something here about canonical cross-sections??? See [AM97], pg. 557.

\section{Boundary Neighborhoods for a Fixed Set of Weights} \label{Boundary Neighborhoods for a Fixed Set of Weights}
In this section, we review results from \cite{avner_cohomology_1997} which we will generalize in later sections. Throughout, we fix a (possibly nontrivial) set of weights for $\Gamma\subset \Gamma_0$. Fix a $\mathbb{Q}$-flag $\mathscr{F}=\{0=V_0\subsetneq V_1\subsetneq \cdots \subsetneq V_\ell=k^n\}$ and let
\begin{equation*}
    \Tilde{W}_\mathscr{F}=\{f\in \tilde{W} \; \; : \;\; S\cdot (M(f(L_0))\cap f(V_j))=f(V_j)\}.
\end{equation*}
This is the set of all marked lattices in $\tilde{W}$ that are well-rounded in each linear subspace $f(V_j)\subseteq S^n$ \cite[562]{avner_cohomology_1997}. Given our choice of $\mathscr{F}$ and a choice of marked lattice, we have a corresponding orthogonal scaling group $\mathscr{U}_{\mathscr{F},f}$ \cite[557]{avner_cohomology_1997}. Decompose $S^n$ as
\begin{equation*}
   S^n=\bigoplus^k_{j=0}V'_j
\end{equation*}
where $V'_j$ is the orthogonal complement of $f(V_{j-1}\cap L_0)\otimes_\mathbb{Q}\mathbb{R}$ in $f(V_{j}\cap L_0)\otimes_\mathbb{Q}\mathbb{R}$. Then, $\mathscr{U}_{\mathscr{F},f}$
consists of those elements of $\textup{End}_{S}(S^n)$ that act by positive real homotheties on each $V'_j$ modulo the full group of homotheties $H\subset G$. We have the following result about this scaling group.
\begin{lemma} [\cite{avner_cohomology_1997}, Lemma~7.3] \label{ash84lemma73}
There exists a $t=(t_1,\cdots, t_{k-1})\in \mathscr{U}_{\mathscr{F},f}$, depending on $f$, such that for any $\rho\in\mathscr{U}_{\mathscr{F},f}$ satisfying $\rho\leq t$, the flag $\mathscr{M}'$ of successive minima for $\rho\cdot f$ contains $\mathscr{F}$.
\end{lemma}

Here, the flag of successive minima $\mathscr{M}'$ depends on the scaled marked lattice $\rho\cdot f$ and on the well-rounded retraction (and hence on the choice of weights used to construct it). Given this lemma, we can construct neighborhoods $N_\mathscr{F}$ near the Borel-Serre boundary face $e(P)$, where $\mathscr{F}$ is the $\mathbb{Q}$-flag associated to the parabolic $P$. Consider the canonical cross-section $C_1=U_pM_p\cdot 1\cdot HK$, as defined in \cite{saper_tilings_1997}. Let $C'_1$ denote the set of all marked lattices representing points in $C_1$. For a given $f_1\in C'_1$, choose $t_1\in \mathscr{U}_{\mathscr{F},f}$ as in Lemma~\ref{ash84lemma73} and define
\begin{equation*}
    N_\mathscr{F}=\bigcup_{f_1\in C'_1}\{\rho_1\cdot f_1 \: \: : \;\; \rho_1<t_1, \; \rho_1\in \mathscr{U}_{\mathscr{F},f_1}\}.
\end{equation*}
Recall that as a set, $e(P)$ corresponds to $X\times^{A_P}\{\infty\}^{\ell-1}$ in the corner $X(P)=X\times^{A_P}\overline{A}_P$, which fibres over $X$, and that the partial compactification $\overline{A}_P$ is isomorphic to $(0,\infty]^{\ell-1}$.
\begin{prop} [\cite{avner_cohomology_1997}, Prop.~7.4] \label{am97prop74}
On the intersection $N_\mathscr{F}$ with a fibre of $X(P)$, the well-rounded retraction is constant and satisfies $r(N_\mathscr{F})=\Tilde{W}_\mathscr{F}$.
\end{prop}

We now extend everything to the Borel-Serre bordification. For marked lattices $f_1\in C'_1$ representing points in the canonical cross-section, there are isomorphisms $\mathscr{U}_{\mathscr{F},f_1}\cong A_P$, which extend to the partial compactifications to give isomorphisms $\overline{\mathscr{U}}_{\mathscr{F},f_1}\cong \overline{A}_P$, where in $\overline{\mathscr{U}}_{\mathscr{F},f_1}$ one allows elements with some zeros on the diagonal \cite[566]{avner_cohomology_1997}. We can then define the following tubular neighborhoods (with boundary) of the boundary faces $e(P)$ in $\overline{X}$
\begin{equation*}
    \overline{N}_\mathscr{F}=\bigcup_{f_1\in C_1}\{\rho_1\cdot f_1 \; \; : \;\; \rho_1<t_1, \; \rho_1\in \overline{\mathscr{U}}_{\mathscr{F},f_1}\},
\end{equation*}
as in \cite[Def.~7.5]{avner_cohomology_1997}. Each $\overline{N}_\mathscr{F}$ is homeomorphic to $e(P)\times (1,\infty]^{\ell-1}$. Recall that $e(P)$ is the quotient of the corner $X(P)$ by $A_P$ under its geodesic action on $X$ and that there is a quotient map $q_P: X(P)\rightarrow e(P)$ of fibre bundles \cite[557]{avner_cohomology_1997}, for which the canonical cross-sections are sections taking the form $\{\cdot\}\times e(P)$ \cite[175]{saper_tilings_1997}. Indeed $X(P)$ is a fibre bundle with $\overline{A}_P\cong (0,\infty]^{\ell-1}$ fibres, and $\overline{N}_\mathscr{F}\cong e(P)\times (1,\infty]^{\ell-1}$ is a subset of $X(P)$. Thus, $\overline{N}_\mathscr{F}$ also has $\overline{A}_P$ fibres, which is to say fibres of the form $\{x\}\times (1,\infty]^{\ell-1}$, where $x\in e(P)$.

The main construction of \cite{avner_cohomology_1997} shows that the well-rounded retraction $r$ extends continuously to a map $\overline{r}:\overline{X}\rightarrow \tilde{W}$, which on the interior of $X$ is just $r$, itself, and which at a point $x\in \overline{X}-X$ is given by the constant value $r(F_x\cap X)$, where $F_x$ is the fibre at $x$ of the tubular neighborhood with boundary $\overline{N}_\mathscr{F}$. This result is summarized in the following proposition, which it will be our goal in later sections to promote to a statement that accounts for a change of weights.

\begin{prop} [\cite{avner_cohomology_1997}, Prop.~7.5] \label{am97prop75}
The map $\overline{r}$ is a continuous extension of $r$ and is a $\Gamma_0$-equivariant retraction. It is constant on the $\overline{A}_P$ fibres of $\overline{N}_\mathscr{F}$.
\end{prop}

%%All of this can be rewritten for the more general context of a reductive algebraic group defined over the rationals and division algebras in the style of Ash.

\section{The Well-Tempered Complex} \label{The Well-Tempered Complex}
Consider the sub-semigroup $\Delta=\{a\in G\; | \; L_0a\subseteq L_0\}$ of $G$. Choose $a\in \Delta$ and let $\Gamma_a=\Gamma_0\cap a^{-1}\Gamma_0 a$.  Our goal is to compute the Hecke operator $T_a$ on equivariant cohomology.  Fix a set of weights $\Phi$ for $\Gamma_0$ (we will suppress the superscript $L_0$ in our notation, understanding $L_0$ to be fixed once and for all) satisfying $\Phi(xa)=\Phi(x)$ for all $x\in L_0\setminus\{0\}$. The well-tempered complex will depend upon this choice of the data $\Phi$, $a$, and $\Gamma_a$. The first step is to construct a one-parameter family of weights $\Phi_\tau: (L_0-\{0\})\times I\rightarrow \mathbb{R}_{>0}$ as follows. Letting $M_0=L_0a$, set
\begin{equation*}
    \Phi_\tau(x)=\left\{\begin{array}{cl}
\Phi(x) & \textup{if}\; x\in M_0-\{0\} \\
\tau^2\cdot\Phi(x) & \textup{if}\; x\notin M_0.
\end{array}\right.
\end{equation*}

Fix the interval $I=[1,\tau_0]$ for $\tau_0>1$ (when working henceforth with the well-tempered complex, we will always take $\tau_0$ to be chosen as in \cite[Thm.~4]{mcconnell_computing_2020}). By \cite[Thm.~2]{mcconnell_computing_2020}, there is a retraction 
\begin{equation*}
    R_\tau(t): (Y\times [1,\tau_0])/K\times [0,1]\longrightarrow (Y\times [1,\tau_0])/K,
\end{equation*}
continuous in both $t$ and $\tau$, which at every temperament $\tau$ is a strong deformation retraction of $(Y\times \{\tau\})/K$ onto $W_{\tau}$, the well-rounded retract defined by the set of weights specified above for $\tau$ fixed. In particular, $W_{\tau}$ is the image of $R_\tau(1)$ for every $\tau\in [1,\tau_0]$.

\begin{defn}[\cite{mcconnell_computing_2020}, Def.~4]
The well-tempered complex $W^+$ for the data $\Phi$, $a$, and $\Gamma_a$ is the image in $(Y\times [1,\tau_0])/K$ of the map $R_\tau(t)$.
\end{defn}

This is a finite cell complex \cite[Thm.~3]{mcconnell_computing_2020}. In what follows, we would like to work initially inside the symmetric space $X$, rather than ``downstairs'' modulo $\Gamma_0$. We can lift the map $R_\tau(t)$ from $Y/K$ up to $X$ via the quotient $X\rightarrow \Gamma_0\backslash X=Y/K$, yielding a $\Gamma_0$-equivariant map
\begin{equation*}
    \Tilde{R}_\tau(t): (X\times [1,\tau_0])\times [0,1]\rightarrow X\times [1,\tau_0].
\end{equation*}
We have fibrewise identifications
\begin{equation*}
    (Y\times \{\tau\})/K\cong (Y/K)\times \{\tau\}=(\Gamma_0\backslash X)\times \{\tau\},
\end{equation*}
coming from the bundle isomorphism $(Y\times [1,\tau_0])/K\cong (Y/K)\times [1,\tau_0]$ (see \cite[8]{mcconnell_computing_2020}). By \cite[Thm.~2]{mcconnell_computing_2020}, $\Tilde{R}_\tau(t)$ is a $\Gamma_0$-equivariant map, continuous in both $\tau$ and $t$. By \cite{avner_cohomology_1997}, we have the well-rounded retractions $r_\tau: X\rightarrow \Tilde{W}_\tau$ on each $\tau$-fibre, which extend as retractions $\overline{r}_\tau: \overline{X}\rightarrow \Tilde{W}_\tau$, where $\Tilde{W}_\tau$ is the inverse image of $W_\tau$ under the quotient map modulo $\Gamma_0$. 

Thus, we have a priori an extension of the map
\begin{equation*}
    \tilde{R}_\tau(1): X\times [1,\tau_0]\rightarrow \tilde{W}^+=\textup{im}(\tilde{R}_\tau(t))
\end{equation*}
to a retraction $\overline{R}_\tau:\overline{X}\times [1,\tau_0]\rightarrow \tilde{W}^+$. Here, as in \cite[(4.3)]{mcconnell_computing_2020}, $\tilde{W}^+$ is the universal ramified cover of $W^+$, which is a regular cell complex admitting a $\Gamma$-action with finite stabilizers of cells \cite[Thm.~3]{mcconnell_computing_2020}. The goal of the following sections will be to show that this map $\overline{R}_\tau$ comes from a deformation retraction of $\overline{X}\times [1,\tau_0]$ onto $\tilde{W}^+$.

\begin{remark}
\normalfont
The complex $\Tilde{W}^+$ can be written set-theoretically as the union
\begin{equation*}
    \bigcup_{\tau\in [1,\tau_0]}(\tilde{W}_\tau\times \{\tau\})\subseteq\mathcal{Q}\times I,
\end{equation*}
which is naturally embedded inside the cone $\mathcal{Q}$ of positive definite symmetric quadratic forms cross the interval $I=[1,\tau_0]$.
\end{remark}

\section{The Borel-Serre Boundary} \label{The Borel-Serre Boundary}

Specifying a set of weights changes the flag of successive minima associated to a given marked lattice, since it changes the definition of ``well-rounded." Likewise, the sets $\Tilde{W}_\mathscr{F}$ change under a continuous one-parameter family of weights, as do their quotients $W_\mathscr{F}$ downstairs modulo $\Gamma_0$. Each $W_\mathscr{F}$ is a closed subcomplex of $W$, and one expects the cell complex structure to change at critical temperaments, as one passes through the well-tempered complex. However, the orthogonal scaling group associated to a marked lattice does not change if we change the weights. Although the new weighted inner products change which vectors are deemed minimal, the notion of orthogonality that gives the decomposition of $S^n$ into the spaces $V'_j$ on which $\mathscr{U}_{\mathscr{F},f}$ acts is unchanged (the inner product defined on $S^n$ using the archimedean values does not depend on the choice of weights). In particular, we still have an isomorphism of groups $A_P\cong \mathscr{U}_{\mathscr{F},f}$ \cite[558]{avner_cohomology_1997}.

%Lemma 7.1 of [AM97] is unaffected under a change in weights. Corollary 7.2 is also unchanged. However, it is the case that Lemma 7.3 furnishes different $t$ for different well-rounded retracts. This is what makes the difference. All of the tubular neighborhoods with boundary $N_\mathscr{F}$ are homeomorphic to $e(P)\times [0,1)^{k-1}$, but they are not all this ``big,'' and they are in general quite jagged, in the sense that for each $f_1$ we have a different $t_1$, both of which depend on the choice of weights.
%

We now explain how to construct a family of boundary neighborhoods, in the spirit of those constructed in Section~\ref{Boundary Neighborhoods for a Fixed Set of Weights}, which vary continuously with the parameter $\tau$. This will account for the changing weights in the definition of the well-tempered complex $\tilde{W}^+$. In a fibre over a fixed $\tau$ of the trivial fibration $\overline{X}\times [1,\tau_0]\rightarrow [1,\tau_0]$, consider the canonical cross-section $C_1=U_pM_p\cdot 1\cdot HK$ as in Section~\ref{Boundary Neighborhoods for a Fixed Set of Weights}. Again, let $C'_1$ denote the set of all marked lattices representing points in $C_1$. For a given $f_1\in C'_1$, choose $t_1(\tau)\in \mathscr{U}_{\mathscr{F},f}$ as in Lemma~\ref{ash84lemma73}. Note that this choice of $t_1$ depends upon the weights and therefore implicitly on $\tau$. This suggests the following definition.

\begin{defn}
At a temperament $\tau\in [1,\tau_0]$, define $\overline{N}_\mathscr{F}(\tau)\subset \overline{X}\times \{\tau\}$ to be the tubular neighborhood with boundary
\begin{equation*}
    \overline{N}_\mathscr{F}(\tau)=\bigcup_{f_1\in C_1}\{\rho_1\cdot f_1 \; \; : \;\; \rho_1<t_1(\tau), \; \rho_1\in \overline{\mathscr{U}}_{\mathscr{F},f_1}\}.
\end{equation*}
\end{defn}

Denote by $N_\mathscr{F}(\tau)$ the corresponding neighborhood inside $X$, where $\overline{\mathscr{U}}_{\mathscr{F},f_1}$ is replaced with $\mathscr{U}_{\mathscr{F},f_1}$. The content of the following lemma is that the neighborhoods $\overline{N}_\mathscr{F}(\tau)$ vary continuously in $\tau$.

\begin{lemma}
The elements $t(\tau)=(t_1(\tau),\cdots, t_{k-1}(\tau))\in \mathscr{U}_{\mathscr{F},f}$ vary continuously in $\tau\in [1,\tau_0]$ with respect to the standard topology on $\mathbb{R}^{\ell-1}$.
\end{lemma}

\begin{proof}
Let $j\in \{1,\cdots, \ell-1\}$ and set $d_j=\textup{dim}_\mathbb{Q}(V_j)$. Let $L^\circ=f(V_j\cap L_0)$, which is a $\mathbb{Z}$-lattice in the $S$-submodule $V^\circ=f(V_j\otimes_\mathbb{Q}\mathbb{R})\subseteq S^n$. These dimensions and subspaces remain unchanged as $\tau$ varies, for a fixed marked lattice $f$, a fixed $\mathbb{Z}$-lattice $L_0$, and a fixed flag $\mathscr{F}$. Perform the well-rounded retraction on $L^\circ$ within $V^\circ$. This means multiplying by constants $\mu_1(\tau),\cdots, \mu_{d_j-1}(\tau)\leq 1$ on the corresponding subspaces of $V^\circ$, after first rescaling $f$ by a homothety. Rescaling does not affect the value of the $\mu_i$, and as in paragraph 2 of the proof of \cite[Thm.~2]{mcconnell_computing_2020}, the action of positive real homotheties is a continuous function of $\tau$, because the family of weights $\varphi_\tau$ is continuous in $\tau$ and the weights are positive. 
 
Define $\beta_j(\tau)=\mu_1(\tau)\cdots\mu_{d_j-1}(\tau)$. Project $f(L_0)$ orthogonally onto $f(V_j)^\perp$ to obtain a lattice $L^\dagger$. Let $\alpha_j(\tau)>0$ be the length of the shortest nonzero vector in $L^\dagger$. Define $t_1(\tau)=\textup{min}(1,(1/2)\alpha_1(\tau)\beta_1(\tau))$. Similarly, define 
\begin{equation*}
    t_j=\textup{min}(1,(1/2)(t_1(\tau)\cdots t_{j-1}(\tau))^{-1}\alpha_j(\tau)\beta_j(\tau))
\end{equation*}
inductively, so that $(t_1(\tau)\cdots t_j(\tau))^{-1}\geq 2(\alpha_j(\tau)\beta_j(\tau))^{-1}$ for all $1\leq j\leq \ell-1$. So far, we have followed the proof of \cite[Lemma 7.3]{avner_cohomology_1997}, nearly word for word, in order to establish notation. We now prove what we want. Recall that we have a one-parameter family of weights for $L_0$ given by the map
\begin{equation*}
    \Phi_\tau:(L_0-\{0\})\times [1,\tau_0]\rightarrow\mathbb{R}_+,
\end{equation*}
which is real-analytic in $\tau$ for any given $x$, and which is normalized continuously in $\tau$ for each $\tau$ so that $\Phi_\tau$ has maximum $1$ for all $\tau$. 

The length $\alpha_j(\tau)$ is that which minimizes $\Phi_\tau(x)\left \langle x,x \right \rangle$ for all $x\in L^\dagger-\{0\}$. Hence, $\alpha_j(\tau)$ varies continuously in $\tau$ for each $1\leq j\leq \ell-1$, since $\Phi_\tau$ is continuous in $\tau$. Now, since $\beta_j(\tau)=\mu_1(\tau)\cdots\mu_{d_j-1}(\tau)$, in order to show that $\beta_j(\tau)$ varies continuously in $\tau$ for each $j$, it suffices to show that for each $j$, the constants $\mu_1(\tau),\cdots, \mu_{d_j-1}(\tau)$ vary continuously in $\tau$, but this is exactly the statement of \cite[Thm.~2]{mcconnell_computing_2020} that $R_\tau(t)$ is continuous in $\tau$. Thus, $t_j(\tau)$ varies continuously in $\tau$ for each $j$, and therefore the elements $t(\tau)\in \mathscr{U}_{\mathscr{F},f}$ form a continuous one-parameter family in $\tau$, as was to be shown.
\end{proof}

Thus, we have constructed a \textit{continuous} one-parameter family of tubular neighborhoods with boundary. Note that one can define analogously, depending on each temperament $\tau$, sets $\Tilde{W}_\mathscr{F}(\tau)$. Applying \cite[Prop.~7.4]{avner_cohomology_1997} at each temperament, we see that the well-rounded retraction is constant on the $\overline{A}_P$ fibres of these neighborhoods. We state this as a corollary of our construction.

\begin{corollary} \label{corIntersectNbarAP}
On the intersection of $N_\mathscr{F}(\tau)$ with an $\overline{A}_P$ fibre of $X(P)$ for a fixed $\tau$, the map $\tilde{R}_\tau(t)$ is constant in $\tau$ and satisfies $\tilde{R}_\tau(t)(N_\mathscr{F}(\tau))=\Tilde{W}_\mathscr{F}(\tau)$.
\end{corollary}

In effect, this corollary provides the appropriate extension of Proposition~\ref{am97prop74} to the setting of the well-tempered complex. Our next task will be to extend Proposition~\ref{am97prop75}.

\section{A Family of Central Tiles}\label{A Family of Central Tiles}
We begin by recalling facts on tilings of symmetric spaces from Section~2 of \cite{saper_tilings_1997}. Denote by $\mathscr{P}$ the set of $\mathbb{Q}$-parabolic subgroups of $G$. Given any $P\in \mathscr{P}$, we can form $\overline{A}_P$ as in Section~\ref{The Borel-Serre Compactification}, and $\overline{A}_P$ acts on $\overline{X}$ by the geodesic action \footnote{Since our convention is for $K$ to act on $G$ on the right, the natural choice is the $NAK$ decomposition of $G$, whence the geodesic action of $A$ is a left action. This is consistent with Saper's convention in \cite{saper_tilings_1997}.}. We shall denote this action by the symbol $\mathbf{o}$. Denote by $\overline{A}_P(1)$ the strictly dominant cone for the pair $(\overline{A}_P,\Delta)$, which is an acute cone defined by the stipulation that an element $x\in \overline{A}_P$ is strictly dominant if $x^{\alpha}>1$ for all $\alpha\in \Delta$ (our fixed system of simple roots, as in Section~\ref{The Borel-Serre Compactification}). With these preliminaries, we have the following definition.

\begin{defn}[\cite{saper_tilings_1997}, Def.~2.1] \label{defSaperTile}
A \textit{tiling} of $\overline{X}$ is a cover by a collection of disjoint sets $\overline{X}_P$, one for each $P\in\mathscr{P}$, such that the following properties are satisfied:
\begin{enumerate}
    \item For the maximal parabolic $P=G$, the set $\overline{X}_G$ is a closed submanifold of the interior $X\subset \overline{X}$ with $\textup{dim}(\overline{X}_G)=\textup{dim}(X)$.
    \item The closed boundary faces of $\overline{X}_G$ form a set $\left \{ \partial^P\overline{X}_G \right \}_{P\in\mathscr{P}}$ for which the bijection $P\mapsto \partial^P\overline{X}_G$ preserves inclusions.
    \item For each closed boundary face $\partial^P\overline{X}_G$, there is a $b_P\in A_P$ and a corresponding canonical cross-section $\left \{ b_P \right \}\times e(P)$ such that $\partial^P\overline{X}_G\subseteq \left \{ b_P \right \}\times e(P)$.
    \item For each $P\in \mathscr{P}$, the set $\overline{X}_P$ is swept out by the geodesic action of $\overline{A}_P(1)$ on $\partial^P\overline{X}_G$. That is, $\overline{X}_P=\partial^P\overline{X}_G\;\mathbf{o}\;\overline{A}_P(1)$.
\end{enumerate}
\end{defn}

It will be necessary for us to make a clever choice of tiling, varying with the temperament $\tau$. As such, we require several facts from \cite{saper_tilings_1997} concerning the existence of tilings. Denote by $\mathscr{P}_1\subset \mathscr{P}$ the set of maximal $\mathbb{Q}$-parabolic subgroups of $G$.
\begin{defn}[\cite{saper_tilings_1997}, Def.~2.6]
The parameter of a tiling is a collection $\textbf{b}=(b_P)_{P\in \mathscr{P}_1}$ of positive real numbers satisfying Property 2 of Definition~\ref{defSaperTile}.
\end{defn}

In particular, the group $G$ determines a space of all parameters $\mathscr{B}=\prod_{Q\in \mathscr{P}_1}A_Q\cong (\mathbb{R}_{>0})^{\mathscr{P}_1}$ and a tiling (if it exists) is uniquely determined by a choice of parameter $\mathbf{b}\in\mathscr{B}$ (\cite{saper_tilings_1997}, comment at the bottom of page 176). Note that one determines the canonical cross-sections $\{b_P\}\times e(P)$ for a given $P\in\mathscr{P}-\mathscr{P}_1$ by intersecting the canonical cross-sections of all maximal parabolics containing $P$ (\cite{saper_tilings_1997}, Remark~1, page 176).

At each temperament $\tau$, we need to pick a tiling of the fibre over $\tau$ for the projection onto the second factor $\overline{X}\times [1,\tau_0]\rightarrow [1,\tau_0]$, and we need to do so in such a way that the central tiles of these tilings vary continuously in $\tau$. This will allow us to mirror the argument in Section~8 of \cite{avner_cohomology_1997} to prove that $\tilde{R}_\tau(t)$ extends to a deformation retraction. While varying continuously in $\tau$, we will see that we can in fact choose a family of tilings that is constant in $\tau$, so long as we choose our parameter to be sufficiently large.

For each $\tau$, we define a neighborhood $B(\tau)$ of $\partial \overline{X}$ in $\overline{X}$, such that $B(\tau)\times \{\tau\}$ is a neighborhood of $\partial \overline{X}\times \{\tau\}\subset \overline{X}\times [1,\tau_0]$. Our definition generalizes that of the neighborhood $B$ in \cite[(8.2)]{avner_cohomology_1997}. We take a union over the finitely many $\mathbb{Q}$-flags $\mathscr{F}$ representing classes of $\mathbb{Q}$-flags modulo $\Gamma$, and define
\begin{equation*}
    B(\tau)=\bigcup_\mathscr{F}\bigcup_{\gamma\in \Gamma}\gamma\cdot \overline{N}_\mathscr{F}(\tau).
\end{equation*}

Then, for each temperament $\tau$, we continue as in \cite[(8.2)]{avner_cohomology_1997} and choose a central tile $X_0(\tau)\subseteq \overline{X}$ such that
\begin{enumerate}
    \item $\partial^PX_0(\tau)\subseteq \bigcup_{\mathscr{F'}}\bigcup_{\gamma\in \Gamma\cap P}(\gamma\cdot \overline{N}_{\mathscr{F'}}(\tau))$ for any proper $\mathbb{Q}$-parabolic subgroup $P$ with associated flag $\mathscr{F}$. Here, we are taking a finite union over $\mathscr{F'}$ representing that $\Gamma\cap P$ classes of $\mathbb{Q}$-flags $\mathscr{F'}\supseteq \mathscr{F}$.
    \item $\partial X_0(\tau)\subseteq B$ and $X_0(\tau)\cup B=\overline{X}$.
\end{enumerate}
We omit at this stage the third condition of \cite[(8.2)]{avner_cohomology_1997}, which requires that $\Tilde{W}_\tau\subseteq \textup{int}(X_0(\tau))$. Instead, we choose a one-parameter family $X_0(\tau)$ of central tiles satisfying only conditions~(1) and~(2), as above.

\begin{prop} \label{propSaperForTau}
There exists a subset $\hat{X}_0\subseteq \overline{X}$ such that for each $\tau\in [1,\tau_0]$, there exists a tiling of $\overline{X}\times \{\tau\}$ with central tile $\hat{X}_0\times \{\tau\}$ satisfying properties~(1) and~(2) above. Moreover, the well-tempered complex $\tilde{W}^+$ is contained in the interior of $\hat{X}_0\times [1,\tau_0]$.
\end{prop}
\begin{proof}
As in \cite{avner_cohomology_1997}, we can give a soft compactness argument saying we can push each central tile $X_0(\tau)$, individually, closer to the boundary if need be. In order to do this, one still needs to enlarge the central tile in a way that preserves the correspondence between its boundary faces and the rational parabolics of $\textup{SL}_n$. The way to do this is to select a larger parameter $\textbf{b}(\tau)$ (recall that \cite{saper_tilings_1997} makes corners by adding $\infty$, not $0$), in our case at each temperament $\tau$. Fortunately, \cite{saper_tilings_1997} proves that tilings always exist so long as the parameter $\textbf{b}(\tau)$ is sufficiently large (and $\Gamma$-invariant), so the enlargement of the central tile proposed in \cite{avner_cohomology_1997} is possible and can be done at each temperament \cite[Thm.~5.7]{saper_tilings_1997}. 
Consider this enlargement of $X_0(\tau)$ for each temperament $\tau$. Then, for each maximal rational parabolic subgroup $P\in\mathscr{P}_1$, we define:
\begin{equation*}
    \hat{b}_P=\sup_{\tau\in [1,\tau_0]}b_P(\tau)\;\;\in \mathbb{R}_{>0}
\end{equation*}
where $b_P(\tau)$ is the corresponding parameter for $P$ at temperament $\tau$ with respect to the enlargement of $X_0(\tau)$. This gives a new set of $\Gamma$-invariant parameters $\textbf{b}=(\hat{b}_P)_{P\in\mathscr{P}_1}$, which we may take uniformly now at each temperament $\tau$ to define the same tiling of $\overline{X}\times \{\tau\}$ for all $\tau$. Denote by $\hat{X}_0$ the common central tile of this tiling. By construction, $\hat{X}_0\times [1,\tau_0]$ now defines a one-parameter family of central tiles that contains the whole well-tempered retract, and at each $\tau$, $\hat{X}_0$ satisfies properties 1 and 2.
\end{proof}

\section{Extending to the Boundary}\label{Extending to the Boundary}
We have completed the necessary geometric work to prove that there exists a deformation retraction of $\overline{X}\times [1,\tau_0]$ onto the well-tempered complex $\tilde{W}^+$, which extends the map $\tilde{R}_\tau(1):X\times [1,\tau_0]\rightarrow \tilde{W}^+$ of Section~\ref{The Well-Tempered Complex}. Consider the map
\begin{equation*}
    \overline{R}_\tau(t):(\overline{X}\times [1,\tau_0])\times [0,1]\rightarrow \overline{X}\times [1,\tau_0]
\end{equation*}
defined as follows. First, we declare that $\overline{R}_\tau(t)(x)=\tilde{R}_\tau(t)(x)$ for all $x\in X\times [1,\tau_0]$. Now, suppose that $\bar{x}\in e(P)$ for a unique $P$ corresponding to a flag $\mathscr{F}$. Then, $\bar{x}\in \overline{N}_\mathscr{F}(\tau)$ for some $\tau$, by Lemma~\ref{ash84lemma73}, and there is a unique fibre $F_{\bar{x}}$ of $\overline{N}(\tau)_\mathscr{F}$ on which $\bar{x}$ lies. Let $\overline{R}_\tau(t)(\bar{x})=\tilde{R}_\tau(t)(F_{\bar{x}}\cap X)$ be the common value of the retraction on this fibre. This is well defined by Corollary~\ref{corIntersectNbarAP}. In particular, we have:
\begin{prop} \label{barR}
The map $\overline{R}_\tau(t)$ is a $\Gamma_0$-equivariant retraction onto $\tilde{W}^+$, continuous in both $\tau$ and $t$, and constant on the $\overline{A}_P$-fibres of $\overline{N}_\mathscr{F}(\tau)$ for each $\tau$.
\end{prop}
\begin{proof}
Restricted to each temperament $\tau$, the map $\overline{R}_\tau(t)$ is a $\Gamma_0$-equivariant retraction \cite[Prop.~8.3]{avner_cohomology_1997}. Combining Corollary~\ref{corIntersectNbarAP} above with \cite[Thm.~2]{mcconnell_computing_2020} shows that $\overline{R}_\tau(t)$ is continuous in $\tau$ and extends as a retraction in from the boundary.
\end{proof}

By \cite{saper_tilings_1997} and \cite[(8.1)]{avner_cohomology_1997}, there is a family of $\Gamma_0$-equivariant, piecewise-analytic deformation retractions $r_{t,S}(\tau):\overline{X}\times [0,1]\rightarrow \hat{X}_0\times \{\tau\}$ for each $\tau\in [1,\tau_0]$, such that $r_{0,S}=\textup{id}_{\overline{X}}$ and $r_{1,S}(\tau)=r_S(\tau)$, where for each $\tau$, the map $r_S(\tau)$ is defined uniquely by the property that $r_S(\tau)(y\; \mathbf{o}\; a)=y$ for all $y\in \partial^P(\hat{X}_0\times \{\tau\})$ and $a\in \overline{A}_p$, $a\leq 1$. Moreover, for $y\in \partial^P(\hat{X}_0\times \{\tau\})$ with $P$ the smallest rational parabolic satisfying this property, $r_S(\tau)(y')=y$ implies that $y'=y\;\mathbf{o}\; a$ for some $a\in \overline{A}_p$, $a\leq 1$. By Proposition~\ref{propSaperForTau}, this extends to a map
\begin{equation*}
    r_{t,S}: (\overline{X}\times [1,\tau_0])\times [0,1]\rightarrow \hat{X}_0\times [1,\tau_0].
\end{equation*}

Using this construction, we promote the retraction of Proposition~\ref{barR} above to a deformation retraction. Define a map
\begin{equation*}
    D(\tau,t):(\overline{X}\times [1,\tau_0])\times [0,1]\rightarrow \overline{X}\times [1,\tau_0]
\end{equation*}
piecewise by the rule
\begin{equation*}
    D(\tau,t)=\left\{\begin{matrix}
r_{2t,S} & t\in [0,1/2]\\ 
(\overline{R}_\tau(t))(r_S(x)) & t\in [1/2,1]
\end{matrix}\right.,
\end{equation*}
which the following theorem shows to be our desired deformation retraction.

\begin{theorem} \label{thmD}
The map $D(\tau,t)$ is a $\Gamma_0$-equivariant deformation retraction of $\overline{X}\times [1,\tau_0]$ onto $\tilde{W}^+$. Moreover, $D(\tau,1)=\overline{R}_\tau$.
\end{theorem}
\begin{proof}
The first assertion follows from Proposition~\ref{barR} above. The second follows by applying \cite[Thm.~8.4]{avner_cohomology_1997} at each temperament.
\end{proof}

In the next section, we will build a series of spectral sequences, following \cite{avner_cohomology_1997}, that will be essential to the statement and proof of our main theorem. To do so, we will require one more result generalizing that of \cite{avner_cohomology_1997}.

\begin{prop}
Let $P$ be a proper, rational parabolic subgroup with associated flag $\mathscr{F}$. Then, for each $\tau$, the map $\overline{R}_\tau$ induces a $(\Gamma_0\cap P)$-equivariant homotopy equivalence between $\bar{e}(P)$ and $\Tilde{W}_\mathscr{F}$.
\end{prop}
\begin{proof}
Apply \cite[Thm.~8.5]{avner_cohomology_1997} at each temperament.
\end{proof}

\section{Cohomology at Infinity}\label{Cohomology at Infinity}
We set up notation, following \cite{avner_cohomology_1997}. Let $\Gamma$ denote any finite-index subgroup of $\Gamma_0$. Denote by $\Phi_l$ the finite set of representatives of $\Gamma$-equivalence classes of $\mathbb{Q}$-flags $\mathscr{F}$ of length $l$. The maximal $\mathbb{Q}$-parabolic subgroups are parametrized by the set $\mathscr{E}=\{(\Gamma\cap P)\backslash\bar{e}(P)\; | \; \mathscr{F}\in\Phi_2\}$. Recall that $\bar{e}(P)$ corresponds to the set $\overline{X}\times^{A_P}\{\infty\}^{l-1}$ in the respective corner of the Borel-Serre bordification $\overline{X}$, where the union of the $\Bar{e}(P)$ form as a set the ``border'' of $\overline{X}$. Thus, we see that for $\mathscr{F}\in\Phi_2$ one has a covering map $(\Gamma\cap P)\backslash\bar{e}(P)\rightarrow \Gamma \bs \partial\overline{X}$. Here, $\Gamma\cap P$ has finite index in $\Gamma$. In particular, this realizes $\mathscr{E}$ as a cover of $\Gamma\bs\partial\overline{X}$.

Letting $\Gamma$ act trivially on any subinterval $I\subseteq [1,\tau_0]$ (we allow also the degenerate case of $I$ a single point), we can define a Čech complex of singular $q$-cochains for all $q\geq 0$ and all $0\leq p\leq n-2$
\begin{equation*}
    \mathscr{X}^{p,q}_I=\bigoplus_{\mathscr{F}\in \Phi_{p+2}}C^q((\Gamma\cap P)\backslash(\bar{e}(P)\times I)).
\end{equation*}
This is a first-quadrant double complex with vertical differential $(-1)^p$ times the usual coboundary maps and horizontal differential the chain map (with appropriate signs) induced by the inclusions $\bar{e}(P')\hookrightarrow \bar{e}(P)$. As in \cite{avner_cohomology_1997}, which follows [BT82], the total complex $\mathscr{X}^{p,q}_I$ computes the singular cohomology
\begin{equation*}
   H^\ast(\Gamma\backslash(\partial\overline{X}\times I))=H^\ast((\Gamma\backslash(\partial\overline{X}))\times I)=H^{\ast}(\Gamma\backslash\partial\overline{X}).
\end{equation*}

In particular, considering the filtration $\{\mathscr{X}^{p,q}|p\geq r\}$ for some fixed $r$, we get a type-II spectral sequence $E^{p,q}_{r,\mathscr{X}_I}$ for $\mathscr{X}^{p,q}_I$ abutting to the cohomology of $\mathscr{X}^{p,q}_I$ in such a way that
\begin{equation*}
    E^{p,q}_{1,\mathscr{X}_I}=\bigoplus_{\mathscr{F}\in \Phi_{p+2}}H^q((\Gamma\cap P)\backslash(\bar{e}(P)\times I)) \Rightarrow H^{p+q}(\Gamma\backslash\partial\overline{X}).
\end{equation*}
Let $\tilde{W}^I_\mathscr{F}$ be the total space of a fibration over $I$ with fibre $\Tilde{W}_\mathscr{F}(\tau)$ for each $\tau\in I$. We can define a new double complex
\begin{equation*}
    \mathscr{W}^{p,q}_I=\bigoplus_{\mathscr{F}\in \Phi_{p+2}}C^q((\Gamma\cap P)\backslash\tilde{W}^I_\mathscr{F})
\end{equation*}
with corresponding type-II spectral sequence $E^{p,q}_{1,\mathscr{W}_I}$ abutting to $H^{p+q+1}(\mathscr{W}^\ast_I)$, where the vertical and horizontal differentials of the double complex $\mathscr{W}^{p,q}_I$ are combined to form a single complex $\mathscr{W}^\ast_I$.

By \cite[Lemma~9.4]{avner_cohomology_1997}, for the single-column double complex
\begin{equation*}
    \hat{\mathscr{X}}^{p,q}_I=\left\{\begin{matrix}
C^q(\Gamma\backslash(\overline{X}\times I)) & p=0 \\
0 & p>0 \\
\end{matrix}\right.
\end{equation*}
there is a chain map $\hat{\mathscr{X}}^{p,q}\rightarrow \mathscr{X}^{p,q}$ inducing via type-I spectral sequences the canonical restriction map $H^q((\Gamma\backslash\overline{X})\times I)\rightarrow H^q((\Gamma\backslash\partial\overline{X})\times I)$ on the corresponding abutments. Forming the corresponding single-column double complex $\hat{\mathscr
W}^{p,q}_I$, one has a chain map $\psi_I:\hat{\mathscr
W}^{p,q}_I\rightarrow \mathscr
W^{p,q}_I$ inducing a map $\psi^\ast_I$ on abutments. We will sometimes denote these maps by $\psi$ and $\psi^\ast$, respectively, when the relevant $I$ is clear from context.

\begin{theorem}
We have spectral sequences abutting to the cohomology of $\mathscr{W}^\ast_I$ as follows
\begin{equation*}
    E^{p,q}_{1,\mathscr{W}_I}=\bigoplus_{\mathscr{F}\in \Phi_{p+2}}H^q((\Gamma\cap P)\backslash\tilde{W}^I_\mathscr{F}) \Rightarrow H^{p+q}(\mathscr{W}^\ast_I),
\end{equation*}
where $I=[u^{(i-1)},u^{(i)}]$ or $I=\tau$. Moreover, the map $\psi_I$ induces a map of spectral sequences, which on the $E_1$ page coincides with the restriction map $H^q(\Gamma\backslash\Tilde{W}_I)\rightarrow E^{p,q}_{1,\Tilde{W}_I}$.
\end{theorem}
\begin{proof}
The spectral sequence abuts to $H^{p+q}(\mathscr{W}^\ast_I)$ as claimed because $\mathscr{W}^\ast_I$ is by definition the single complex associated to $\mathscr{W}^{p,q}_I$. The second statement follows by the same arguments as Theorem 9.5(ii) of \cite{avner_cohomology_1997}, but with the chain map $\psi$ of \cite{avner_cohomology_1997} replaced with the chain map $\psi_I$, as defined above.
\end{proof}

\section{Hecke Correspondences}\label{Hecke Correspondences}
The Hecke correspondences we wish to study depend upon certain data, which are summarized concisely by the notion of a Hecke pair.

\begin{defn}
A Hecke pair is a pair $(\Gamma_0,\Delta)$, where $\Gamma_0$ is an arithmetic subgroup of $G$ and $\Delta$ is the sub-semigroup of $G$ consisting of all $a\in G$ such that $L_0a\subseteq L_0$.
\end{defn}

For each $a\in \Delta$, we have the finite index arithmetic group $\Gamma_a=\Gamma_0\cap a^{-1}\Gamma_0a$. Let $\Gamma$ be an arbitrary finite index subgroup of $\Gamma_0$. We have a left multiplication map $m_a:X\rightarrow X$ given by $m_a(x)=ax$ and the identity $\textup{id}_X:X\rightarrow X$. These induce the two maps
\begin{equation*}
\begin{tikzcd}
\Gamma_a\backslash X \arrow[d,bend right, "p"'] \arrow[d, bend left, "q"]\\
\Gamma_0\backslash X
\end{tikzcd}
\end{equation*}
given by $p:\Gamma_a g K\mapsto \Gamma_0 g K$ and $q: \Gamma_a g K\mapsto \Gamma_0 a g K$. The natural extensions (see Lemma~\ref{maToBar} below) of $m_a$ and $\textup{id}_X$ to $\overline{X}$ induce extensions of $p$ and $q$ to the corresponding Borel-Serre compactifications $\Gamma_a\backslash\overline{X}$ and $\Gamma_0\backslash\overline{X}$. 

\begin{remark}
\normalfont We will abuse notation to denote the maps induced on the Borel-Serre compactifications again by $p$ and $q$. We will also sometimes restrict the domain of $p$ or $q$ to certain well-rounded retracts at certain temperaments, and when doing so will denote the restrictions again by $p$ and $q$.
\end{remark}

Let $\rho$ be any left $\mathbb{Z}\Delta$-module \cite[(5.1)]{mcconnell_computing_2020}. In Section \ref{Cohomology at Infinity}, we use singular cohomology, as in \cite[(9.2)]{avner_cohomology_1997}. For the remainder of the paper, we switch to equivariant cohomology, as in \cite{mcconnell_computing_2020}, with coefficients in $\rho$. In relating singular and equivariant cohomology, we must be careful about ``bad'' primes (i.e., those which divide the order of some finite subgroup of $\Gamma_0$), of which there are finitely many, since $\Gamma_0$ is arithmetic \cite[(5.1)]{mcconnell_computing_2020}.
If the module underlying $\rho$ is a vector space over a field of characteristic not equal to any of the finitely many bad primes (for instance, a field of characteristic zero), then $H^\ast_{\Gamma}(X;\rho)$ is canonically isomorphic to $H^\ast(\Gamma\backslash X;\rho)$, and to the group cohomology $H^\ast(\Gamma;\rho)$, for any finite index subgroup $\Gamma$ of $\Gamma_0$. All our results hold equally for each kind of cohomology \cite[(10.4)]{avner_cohomology_1997}.

Applying the cohomology functor to the previous diagram, we get the following diagram in cohomology
\begin{equation*}
    \begin{tikzcd}
H^\ast_{\Gamma_a}(\overline{X}) \arrow[d,bend right, "p_\ast"'] \\
H^\ast_{\Gamma_0}(\overline{X})\arrow[u, bend right, "q^\ast"']
\end{tikzcd}
\end{equation*}
yielding the Hecke operator $T_a$ on $H^\ast_{\Gamma_0}(\overline{X})$ given by $T_a=p_\ast q^\ast$. By \cite[Thm.~5]{mcconnell_computing_2020}, we can compute $T_a$ in finite terms as the composition
\begin{equation*}
    T_a=p_\ast \ell^{(0)\ast} r^{(1)}_\ast\ell^{(1)\ast} r^{(2)}_\ast\cdots \ell^{(k-1)\ast} r^{(k)}_\ast q^\ast.
\end{equation*}

We will need the following lemma for the proof of our main theorem.

\begin{lemma} \label{maToBar}
The natural extension of the map $m_a:X\rightarrow X$ to $\overline{X}$ satisfies $m_a(\partial \overline{X})=\partial \overline{X}$.
\end{lemma}
\begin{proof}
The boundary components (also called the ``boundary faces'') of $\overline{X}$ are $e(P)=X\times^{A_P}\{\infty\}^{\ell-1}$, for each proper $\mathbb{Q}$-parabolic subgroup $P\subset G$. The equivalence relation on the corner $X(P)$ given by the geodesic action of $A_P$ on $X$ induces one on the boundary face $e(P)$. Denote by $[x,\infty]$ the equivalence class of a point in $e(P)$, where $(x,\infty)\sim (x',\infty)$ if and only if there exists some $b\in A_P$ such that $x=x'\mathbf{o}\; b$ \cite[(5.1)]{borel_corners_1973}. We extend $m_a$ to $\partial \bar{X}$ by declaring that $m_a[x,\infty]=[ax,\infty]$ for any $[x,\infty]\in e(P)$. We have established the inclusion $m_a(\partial \bar{X})\subseteq\partial \bar{X}$. Equality follows from the fact that $m_a$ is invertible with inverse $m_{a^{-1}}$.
\end{proof}

It is worth noting that the double complexes defined in Section~\ref{Cohomology at Infinity} admit a Hecke action, which is the content of the following lemma. This is at each temperament, but the result generalizes naturally.

\begin{lemma}
$T_a$ acts on the Čech double complexes $\mathscr{X}^{p,q}$ and $\mathscr{W}^{p,q}$, commuting with their total differentials.
\end{lemma}
\begin{proof}
In order to prove that $T_a$ commutes with the total differential of $\mathscr{X}^{p,q}$, we will show that it commutes with the usual coboundary maps on cochains and with the horizontal differential induced by the inclusions $\bar{e}(P')\hookrightarrow \bar{e}(P)$. Although $T_a$ permutes boundary faces (see Lemma~\ref{maToBar}), it preserves their inclusions, so it commutes with the horizontal differential of $\mathscr{X}^{p,q}$. The map $q^\ast$ is natural on cohomology, so commutes with the vertical differential. This is also true for the map $p_\ast$, since $p$ is a projection, hence a finite morphism, and so $p_\ast$ is exact on cochains. Thus, the composition commutes with the vertical differential.
%%Math stack exchange search for the parenthetical justification, here.
The vertical differential of $\mathscr{W}^{p,q}$ is also the usual coboundary map, with which we have just shown $T_a$ to commute. The horizontal differential is induced by the inclusions $\Tilde{W}_{\mathscr{F}'}\hookrightarrow \Tilde{W}_{\mathscr{F}}$ (order reversing from inclusions of rational flags), with which $m_a$ and $\textup{id}_X$ commute (again, both preserve inclusions). Thus, $T_a$ also commutes with the horizontal differentials on $\mathscr{W}^{p,q}$, and hence the total differential.
\end{proof}

%%I don't think we actually need this lemma, but it might be good to state in the background. It should be written out somewhere in the literature, anyway, and I haven't come across it before. So perhaps useful as a public service to include in our paper.

We are now ready to state our main theorem. This result is a generalization of [\cite{avner_cohomology_1997}, (9.5i)] to the context of the well-tempered complex. A homological version of Theorem 9.3 can be proven by essentially the same techniques (c.f. [\cite{avner_cohomology_1997}, (10.2i)]). For convenience, we will adopt the notation $I_{u^{(i)}}=[u^{(i-1)},u^{(i)}]$.

\begin{theorem} \label{mainThmCubes}
At each step of the composition 
\begin{equation*}
    T_a=p_\ast \ell^{(0)\ast} r^{(1)}_\ast\ell^{(1)\ast} r^{(2)}_\ast\cdots \ell^{(k-1)\ast} r^{(k)}_\ast q^\ast
\end{equation*}
we have a commutative cube. There are three types: type $\textbf{L-R}$ involving the left and right inclusions; type $\textbf{P}$ involving the map $p$; and type $\textbf{Q}$ involving the map $q$. In each diagram, the vertical maps in the foreground and background are natural isomorphisms. These commutative diagrams are as follows, listed by type.
\newpage

\begin{center}
    \textbf{Type (L-R)}
\end{center}
\begin{equation*}
   \begin{tikzcd}[row sep=2 em, column sep =0.005em]
    H^\ast_{\Gamma_0}(\partial\overline{X}\times I_{u^{(i)}})\arrow[dd, near start, "\cong"'] \arrow[rr,"\ell^{(i-1)\ast}"] && H^\ast_{\Gamma_a}(\partial\overline{X}\times \{u^{(i-1)}\}) \arrow[dd,near start,"\cong"] \\
    & H^\ast_{\Gamma_0}(\overline{X}\times I_{u^{(i)}}) \arrow[rr,crossing over, near start, "\ell^{(i-1)\ast}"]\arrow[ul, "rest."'] && H^\ast_{\Gamma_a}(\overline{X}\times \{u^{(i-1)}\}) \arrow[dd,"\cong"]\arrow[ul,"rest."']\\
    H^\ast_{\Gamma_0}(\mathscr{W}^\ast_{I_{u^{(i)}}}) \arrow[rr] && H^\ast_{\Gamma_a} (\mathscr{W}^\ast_{u^{(i-1)}}) \\
    & H^\ast_{\Gamma_0}(\tilde{W}_{I_{u^{(i)}}})\arrow[from=uu, crossing over, near start, "\cong"]\arrow[ul, "\psi^\ast"]\arrow[rr, "\ell^{(i-1)\ast}"'] && H^\ast_{\Gamma_a}(\tilde{W}_{u^{(i-1)}})\arrow[ul,"\psi^\ast"']
    \end{tikzcd}
\end{equation*}
\vspace{0.5cm}

\begin{center}
    \textbf{Type (P)}
\end{center}
\begin{equation*}
    \begin{tikzcd}[row sep=scriptsize, column sep = 0.5em]
    H^\ast_{\Gamma_0}(\partial\overline{X}\times \{1\})\arrow[dd, near start, "\cong"'] \arrow[rr, "\textup{id}^\ast_X"] && H^\ast_{\Gamma_0}(\partial\overline{X}\times \{1\}) \arrow[dd, near start, "\cong"] \\
    & H^\ast_{\Gamma_0}(\overline{X}\times \{1\}) \arrow[rr, crossing over, near start, "\textup{id}^\ast_X"]\arrow[ul, "rest."'] && H^\ast_{\Gamma_0} (\overline{X}\times \{1\}) \arrow[dd,"\cong"]\arrow[ul,"rest."']\\
    H^\ast_{\Gamma_0}(\mathscr{W}^\ast) \arrow[rr] && H^\ast_{\Gamma_0} (\mathscr{W}^\ast) \\
    & H^\ast_{\Gamma_0}(\tilde{W})\arrow[from=uu, crossing over, near start, "\cong"]\arrow[ul, "\psi^\ast"]\arrow[rr, "p_\ast"'] && H^\ast_{\Gamma_0}(\tilde{W})\arrow[ul, "\psi^\ast"']
    \end{tikzcd}
\end{equation*}
\vspace{0.5cm}

\begin{center}
    \textbf{Type (Q)}
\end{center}
\begin{equation*}
    \begin{tikzcd}[row sep=scriptsize, column sep = 0.5em]
    H^\ast_{\Gamma_0}(\partial\overline{X}\times \{1\})\arrow[dd, near start, "\cong"'] \arrow[rr, "m^\ast_a"] && H^\ast_{\Gamma_a}(\partial\overline{X}\times \{\tau_0\}) \arrow[dd, near start, "\cong"] \\
    & H^\ast_{\Gamma_0}(\overline{X}\times \{1\}) \arrow[rr, crossing over, near start, "m^\ast_a"]\arrow[ul, "rest."'] && H^\ast_{\Gamma_a} (\overline{X}\times \{\tau_0\}) \arrow[dd,"\cong"]\arrow[ul,"rest."']\\
    H^\ast_{\Gamma_0}(\mathscr{W}^\ast) \arrow[rr] && H^\ast_{\Gamma_a} (\mathscr{W}^\ast_{\tau_0}) \\
    & H^\ast_{\Gamma_0}(\tilde{W})\arrow[from=uu, crossing over, near start, "\cong"]\arrow[ul, "\psi^\ast"]\arrow[rr, "q^\ast"'] && H^\ast_{\Gamma_a}(\tilde{W}_{\tau_0})\arrow[ul, "\psi^\ast"']
    \end{tikzcd}
\end{equation*}

\end{theorem}
\begin{proof}
For convenience, let us call the cubes involving left and right inclusions cubes of type \textbf{L} and \textbf{R}, respectively, those involving the map $q$ cubes of type \textbf{Q}, and those involving the map $p$ cubes of type \textbf{P}. We shall prove that the faces of the cubes of each type commute. Observe first that the vertical isomorphisms in the foreground for the cubes of each type are induced from our deformation retract $D(\tau,1)$ at the appropriate temperament or family of temperaments. Applying Theorem 9.5(i) of \cite{avner_cohomology_1997}, we see that the background and foreground vertical maps are natural isomorphisms and that the diagrams forming the left and right faces of each cube commute. It remains to analyze, case by case, the top, bottom, background, and foreground faces.
\vspace{2mm}

\noindent
\textbf{Type L-R.} We present the proof for the cubes of type \textbf{L}. The proof for the cubes of type \textbf{R} is essentially the same. The top face of each cube of Type \textbf{L} clearly commutes, as we simply have compositions of restriction maps. Consider now the back faces. Our global deformation retraction induces a filtration-preserving chain map $\mathscr{X}^{p,q}_{I}\rightarrow \mathscr{W}^{p,q}_{I}$ and furthermore a map of spectral sequences. Since this is an isomorphism on the $E_1$ page (see \cite{avner_cohomology_1997}, the proof of Theorem 9.5(i)), we know by \cite[Prop.~2.6]{brown_cohomology_1982} that the induced map on cohomology is an isomorphism, which forms the background vertical arrows for the appropriate choice of $I$. This map commutes at the chain level with restriction, so the background square commutes. The vertical maps in the foreground square are isomorphisms induced by the chain map $\hat{\mathscr{X}}^{p,q}_I\rightarrow\hat{\mathscr{W}}^{p,q}_I$ defined by our deformation retraction. This map commutes on the chain level with restriction, so the foreground squares commute. It remains to address the bottom face. From Section~\ref{Extending to the Boundary}, we have a map $\psi: \hat{\mathscr{W}}^{p,q}_I\rightarrow \mathscr{W}^{p,q}_I$, which induces the left and right arrows of the bottom face. Again, this map commutes on the chain level with restriction, so the bottom square commutes.
\vspace{2mm}

\noindent
\textbf{Types P and Q.} We present the proof for the cubes of type \textbf{Q}. The proof is analogous to that for the the cubes of type \textbf{P}, only simpler as the map $p$ is induced by the identity map (which, in particular, does extend continuously to $\overline{X}$). By definition, the map $m_a:\overline{X}\rightarrow\overline{X}$ takes the interior of $\overline{X}$ to itself. Thus, it follows from Lemma~\ref{maToBar} that $m_a$ commutes with restriction to $\partial X$, and so the top face of the corresponding diagram commutes. This demonstrates that the top face commutes. At the chain level for a given temperament $\tau$, the map $\psi$ behaves as
\begin{equation*}
    \omega\in C^q(\Gamma_0\backslash\Tilde{W}_\tau)\mapsto \sum_{\mathscr{F}\in\Phi_2}\omega|_{(\Gamma_0\cap P)\backslash\Tilde{W}_\mathscr{F}(\tau)}\in\bigoplus_{\mathscr{F}\in\Phi_2}C^q((\Gamma_0\cap P)\backslash\Tilde{W}_\mathscr{F}(\tau))
\end{equation*}
for $p=0$ and is the zero map for $p>0$, since the double complex $\hat{\mathscr{W}}^{p,q}$ is by definition zero in those degrees \cite[571--572]{avner_cohomology_1997}. By \cite[Thm.~4]{mcconnell_computing_2020}, we know $m_a(\Tilde{W})=\Tilde{W}_{\tau_0}$ and $q=m_a|_{\Tilde{W}}$ (indeed, $m_a$ gives a homeomorphism of cell complexes from the first to the last temperament; c.f. remark \ref{ssAtTau} below). By the observation in Lemma~\ref{maToBar}, we see that $m_a$ will permute the $\Tilde{W}_\mathscr{F}(\tau)$, but the sum above remains unchanged. Thus, $\psi$ and $q$ commute on the chain level and we conclude that the bottom face commutes. Finally, recall that $m_a(\Tilde{W})=\Tilde{W}_{\tau_0}$, $\Tilde{W}=\textup{im}(D(1,t))$, and $\Tilde{W}_{\tau_0}=\textup{im}(D(\tau_0,t))$, by Theorem~\ref{thmD}. Thus, on the image of $D(1,t)$, the map $q$ is just exactly multiplication by $a$, whose image is $\Tilde{W}_{\tau_0}$, and this is precisely the image of $D(\tau_0,t)$. It follows that the front and back faces commute.
\end{proof}

It is useful to note that the statement of Theorem 9.3 holds if one replaces $\Gamma_0$ with $\Gamma_a$ and $\Gamma_a$ with the finite index subgroup $\Gamma\cap\Gamma_a$.

\begin{remark} \label{ssAtTau}
\normalfont As remarked on \cite[20]{mcconnell_computing_2020}, Theorem~4 of \cite{mcconnell_computing_2020} gives a homeomorphism of cell complexes $m_{a^{-1}}:\tilde{W}_{\tau_0}\rightarrow\tilde{W}_1$ from the last temperament to the first temperament, given by multiplication by $a^{-1}$. Using this homeomorphism, one can promote the fibration $\tilde{W}^+\rightarrow [1,\tau_0]$ to a nontrivial fibration $T(\tilde{W}^+)\rightarrow S^1$ over the circle, where the fibres over the first and last temperaments are glued together via $m_{a^{-1}}$. The corresponding Serre spectral sequence is

\begin{equation*}
    E^{p,q}_2=H^p(S^1, H^q(F))\Rightarrow H^{p+q}(T(\tilde{W}^+)),
\end{equation*}
where $F$ is the homotopy fibre of the fibration $T(\tilde{W}^+)\rightarrow S^1$. Thus, one is free to compute the above spectral sequence taking $F=\tilde{W}_\tau$ for any $\tau\in [1,\tau_0]$.
\end{remark}

\section{Comments on Implementation}\label{Comments on Implementation}

We remark on how to work out the spectral sequences in
Theorem~\ref{mainThmCubes} on a computer.

In \cite[(10.1)]{avner_cohomology_1997}, the primary goal was to
compute the homology and cohomology of $\Gamma\bs \tilde W$ and its
subcomplexes $(\Gamma \cap P)\bs \tilde W_\mathscr{F}$.  These are
finite cell complexes.  The paper proposed to store a cell of $(\Gamma
\cap P)\bs \tilde W_\mathscr{F}$ in the computer as a cell~$\sigma$ of
$\Gamma\bs \tilde W$ together with extra structure depending
on~$\mathscr{F}$.  However, this extra structure is delicate.  For a
given~$\sigma$, more than one flag~$\mathscr{F}$ may satisfy
$\sigma\in \tilde W_\mathscr{F}$, and the associated parabolics could
intersect in a variety of ways.  Programs using these ideas proved to
be error-prone.

Equivariant cohomology is a more natural setting for computations of
group cohomology.
%% , as illustrated for example by \cite[(VII.7)]{brown_cohomology_1982}.
As remarked in Section \ref{Hecke Correspondences}, the coefficient
module $\rho$ is over a field of characteristic zero or of
characteristic~$p$ prime to the order of all the torsion elements
of~$\Gamma$, then the two kinds of cohomology are canonically
isomorphic.  When~$p$ does divide the order of torsion elements
of~$\Gamma$, \cite{AGM5} studied the relations between the two kinds
of cohomology in our setting---that is, when the cells~$\sigma$ belong
to the well-rounded retract for $\textup{SL}_n(\mathbb{Z})$ for
small~$n$.  In~\cite{ash_cohomology_2020}, the authors applied the
well-rounded retract for $\textup{SL}_4(\mathbb{Z})$ to computation of
non-constant (twisted) coefficient systems.  Changing from ordinary to
equivariant cohomology in~\cite{ash_cohomology_2020} made the
programming much more tractable, as we now describe.

In general, equivariant cohomology is computed by a spectral sequence
\cite[(VII.7.7)]{brown_cohomology_1982}.  Each spectral sequence in
Theorem~\ref{mainThmCubes} is, in principle, a spectral sequence of
spectral sequences \cite[(10.4)]{avner_cohomology_1997}.  To simplify matters, assume for the rest of this
section that the characteristic is indeed zero or prime to the torsion
in~$\Gamma$.  Then the equivariant cohomology spectral sequence
reduces at~$E_1$ to a single complex.  This is the approach
of~\cite{ash_cohomology_2020}, and it is how we plan to implement
Theorem~\ref{mainThmCubes}.

One of us (MM) has written two bodies of code for computations with
equivariant cohomology.  The first, begun in 2013, is written in Sage
\cite{sage}.  It implements equivariant cohomology of a group~$\Gamma$
acting on any contractible cell complex~$\tilde W$ with finitely many
cells mod~$\Gamma$ and with finite stabilizers of cells.  The
coefficient module can be any $\mathbb{Z}\Delta$-module on a
finite-dimensional vector space over fields like~$\mathbb{Q}$ and
$\mathbb{F}_p$.  In practice, we have used the code with
appropriate~$\tilde W$ for subgroups of $\textup{SL}_n(\mathbb{Z})$
for $n=2,3,4$ as well as of $\textup{Sp}_4(\mathbb{Z})$.  The
implementation is described in more detail in~\cite{ash_cohomology_2020}.

The Sage code also has a complete set of functions for computing Hecke
operators on $H_\Gamma^*(\tilde W; \mathscr{M})$ by the algorithm
in~\cite{mcconnell_computing_2020}.  See Sections~5--6
of~\cite{mcconnell_computing_2020} for the implementation and
Section~7 for results for a range of $\Gamma \subseteq
\textup{SL}_3(\mathbb{Z})$.

The second body of code is written in Haskell and was begun in 2021.
Haskell offers strong static typing, purity (immutable data), and
non-strict evaluation (a computation is not performed if it is not
needed).  This combination makes it ideal for writing code that looks
and feels like mathematics.

In the Haskell code, we have a new implementation of equivariant
cohomology.  We explain it for $\tilde W$, though it applies in the
same way to every $\tilde W_\tau$ and to $\tilde W^+$.  In
computations for the series of papers \cite{AGM1} through
\cite{ash_cohomology_2020}, the code would begin by choosing a set of
representatives~$\sigma$ of the cells of $\tilde W$ mod~$\Gamma$.  It
would do all its computations using those~$\sigma$, creating tables of
auxiliary data restricted to just those~$\sigma$.  Questions of
which~$\gamma$ in the stabilizer of~$\sigma$ would preserve or reverse
the orientation of~$\sigma$ were especially delicate.  Even worse, the
set of representatives for $(\Gamma\cap P)\bs \tilde W_\mathscr{F}$
will be larger than the set of those representatives for $\Gamma\bs
\tilde W$ which happen to lie in $\tilde W_\mathscr{F}$.  As we
remarked above, it is awkward to fix representatives~$\sigma$ for
each~$\mathscr{F}$ separately.  Instead, in the new implementation, we
store a single finite closed subcomplex of $\tilde W$ called a
\emph{chunk}~$\mathcal{C}$ of~$\tilde W$.  The chunk is large enough
to contain representatives for every orbit mod~$\Gamma$ and at the
same time for every orbit mod $(\Gamma\cap P)$.  

We construct the
chunk by walking through~$\tilde W$ in a breadth-first way: start at
any vertex, find all the top-dimensional cells surrounding that
vertex, find all the vertices of those top cells, find all the top
cells surrounding those new vertices, and so on.  During the
breadth-first walk, we check which cells are equivalent to each other
modulo the various $\Gamma\cap P$.  The walk stops when a complete set
of inequivalent cells has been found.  $\mathcal{C}$ is stored as a
ranked poset of cells, where the partial ordering encodes the face
relations.  The chunk approach makes computations easier because,
being a closed subcomplex, it stores~$\sigma$ together with \emph{all}
$\sigma$'s faces, not just representatives of those faces.  For
example, when $\gamma\in\Gamma$ stabilizes $\sigma$, we can
immediately see how~$\gamma$ acts on the orientation: we simply look
at the elements below~$\sigma$ in the ranked poset and note
how~$\gamma$ permutes them.

In the present paper, when (because of our assumption on the
characteristic) the equivariant cohomology of each $\tilde
W_\mathscr{F}$ is a single complex, the spectral sequences in
Theorem~\ref{mainThmCubes} will come from double complexes.  We will
program up these double complexes, computing their~$d_1$, $d_2$, and
higher differentials.  To compute a Hecke operator~$T_a$, we will use
the methods of \cite[(5.3)]{mcconnell_computing_2020}, simply
replacing single complexes with double complexes.  The
quasi-isomorphisms and $p_*$ and $q^*$ will be maps of spectral
sequences, that is, maps of double complexes which induce maps on the
associated spectral sequences.  The final result will be the
$E_\infty$ page of the equivariant cohomology of~$\Gamma$ together
with its decomposition into Hecke eigenspaces.

%%%%%%%%%%%%%%%%%%%%%%%%%%%%%%%%%%%%%%%%

\nocite{borel_corners_1973}
\nocite{voronoi_sur_1908}
\newpage
\printbibliography

\end{document}